\documentclass[a4paper,12pt,intlimits,oneside]{amsart}
\usepackage{graphicx}

\usepackage{enumerate}
\usepackage{amsfonts}
\usepackage{amsmath,amsthm,amssymb}

\newcommand{\comment}[1]{}

\newtheorem{theorem}{Theorem}

\newtheorem{lemma}[theorem]{Lemma}

\newtheorem{remark}[theorem]{Remark}
\newtheorem{conjecture}[theorem]{Conjecture}

\newcommand\La{\Lambda}
\newcommand\de{\delta}
\newcommand\al{\alpha}

\newcommand\f{\varphi}

\newcommand\dbj{\overline{d}_j}

\newcommand{\NN}{\mathbb N}
\newcommand{\ZZ}{\mathbb Z}
\newcommand{\RR}{\mathbb R}

\newcommand{\TT}{\mathbb T}

\newcounter{rek}
\newcounter{rev}
\newcommand{\rev}[1]{
    \marginpar{\refstepcounter{rev}{\small\therev(Sz.R.): #1}}
}

\newcounter{rem}
\newcounter{rec}
\begin{document}

\title[Hardy-Littlewood majorant problem]{On MOckenhoupt's Conjecture in the
Hardy-Littlewood majorant problem}\thanks{Supported in part by the Hungarian National
Foundation for Scientific Research, Project \#'s K-81658 and K-100461.}

\author{S\'andor Krenedits}

\date{\today}

%%%%
%%%%\address{ A. R\'enyi Institute of Mathematics
%%%%\newline \indent Hungarian Academy of Sciences
%%%%\newline \indent Budapest, P.O.B. 127, 1364 Hungary.}
%%%%\email{krenedits@t-online.hu}
%%%%

%%% \maketitle

\begin{abstract} The Hardy-Littlewood majorant problem has a positive answer only for exponents $p$ which are even integers, while there are counterexamples for all $p\notin 2\NN$. Montgomery conjectured that even among the idempotent polynomials there must exist some counterexamples, i.e. there exist some finite set of characters and some $\pm$ signs with which the signed character sum has larger $p^{\rm th}$ norm than the idempotent obtained with all the signs chosen $+$ in the character sum. That conjecture was proved recently by Mockenhaupt and Schlag.

However, Mockenhaupt conjectured that even the classical $1+e^{2\pi i x} \pm e^{2\pi i (k+2)x}$ three-term character sums, used for $p=3$ and $k=1$ already by Hardy and Littlewood, should work in this respect. That remained unproved, as the construction of Mockenhaupt and Schlag works with four-term idempotents. In our previous work we proved this conjecture for $k=0,1, 2$, i.e. in the range $0<p<6$, $p\notin 2\NN$.

Continuing this work here we demonstrate that even the $k=3,4$ cases hold true. Several refinement in the technical features of our approach include improved fourth order quadrature formulae, finite estimation of $G'^2/G$ (with $G$ being the absolute value square function of an idempotent), valid even at a zero of $G$, and detailed error estimates of approximations of various derivatives in subintervals, chosen to have accelerated convergence due to smaller radius of the Taylor approximation.
\end{abstract}

\maketitle

\vskip1em \noindent{\small \textbf{Mathematics Subject
Classification (2000):} Primary 42A05. \\[1em]
}

%%%%%%%%%%
%%%%%%%%%%
%%%%%%%%%%
%%%%%%%%%%

%%% \newpage

\section{Introduction}\label{sec:intro}

We denote, as usual, $\TT:=\RR/\ZZ$ the one dimensional torus
or circle group. Following Hardy and Litlewood \cite{HL}, $f$ is
said to be a majorant to $g$ if $|\widehat{g}|\leq \widehat{f}$.
Obviously, then $f$ is necessarily a positive definite function.
The (upper) majorization property (with constant 1) is the
statement that whenever  $f\in L^p(\TT)$ is a majorant of $g\in
L^p(\TT)$, then $\|g\|_p\leq \|f\|_p$. Hardy and Littlewood proved
this for all $p\in 2\NN$ -- this being an easy consequence of the
Parseval identity. On the other hand already Hardy and Littlewood
observed that this fails for $p=3$. Indeed, they took
$f=1+e_1+e_3$ and $g=1-e_1+e_3$ (where here and the sequel we
denote $e_k(x):=e(kx)$ and $e(t):=e^{2\pi i t}$, as usual) and
calculated that $\|f\|_3<\|g\|_3$.

The failure of the majorization property for $p\notin 2\NN$ was
shown by Boas \cite{Boas}. Boas' construction exploits complex Taylor
series expansion around zero: for $2k<p<2k+2$ the counterexample
is provided by the polynomials $f,g:=1+ re_1\pm r^{k+2}e_{k+2}$,
with $r$ sufficiently small to make the effect of the first terms
dominant over later, larger powers of $r$.

Utilizing Riesz products -- an idea suggested to him by Y. Katznelson -- Bachelis proved \cite{Bac} the failure of the majorization property for any $p\notin 2\NN$ even with arbitrarily large constants. That is, not even $\|g\|_p < C_p
\|f\|_p$ holds with some fixed constant $C=C_p$.

Montgomery conjectured that the majorant property for $p\notin
2\NN$ fails also if we restrict to \emph{idempotent} majorants,
see \cite[p. 144]{Mon}. (A suitable integrable function is
idempotent if its convolution square is itself: that is, if its
Fourier coefficients are either 0 or 1.) This has been recently
proved by Mockenhaupt and Schlag in \cite{MS}.

\begin{theorem}[{\bf Mockenhaupt \& Schlag}]\label{th:MMS} Let
$p>2$ and $p\notin 2\NN$, and let $k>p/2$ be arbitrary. Then for
the trigonometric polynomials $g:=(1+e_k)(1-e_{k+1})$ and
$f:=(1+e_k)(1+e_{k+1})$ we have $\|g\|_p >\|f\|_p$.
\end{theorem}

Oddly enough, the quite nice, constructive example is given with a
four-term idempotent polynomial, although trinomials may seem
simpler objects to study. Indeed, there is a considerable
knowledge, even if usually for the maximum norm, on the space of
trinomials, see e.g. \cite{Cheb,Neu,Mini}. Note that three-term examples are the simplest we can ask for, as two-term polynomials can never exhibit failure of the majorization property. In the construction of Mockenhaupt and Schlag, however, the key role is played by the fact that the given 4-term idempotent is the product of two two-term idempotents, the $p^{\rm th}$ power integral of which then can be expressed by the usual trigonometric and hyperbolic functions. So even if four terms is a bit more complicated, but the product form gives way to a manageable calculation.

Nevertheless, one may feel that Boas' idea, i.e. the idea of
cancelation in the $(k+1)^{\rm st}$ Fourier coefficients works
even if $r$ is not that small -- perhaps even if $r=1$. The
difficulty here is that the binomial series expansion diverges,
and we have no explicit way to control the interplay of the
various terms occurring with the $\pm$ signed versions of our
polynomials. But at least there is one instance, the case of
$p=3$, when all this is explicitly known: already Hardy and
Littlewood \cite{HL} observed that failure of the majorant
property for $p=3$ is exhibited already by the pair of idempotents
$1+e_1\pm e_3$. In fact, this idempotent example led Montgomery to
express (in a vague form, however, see \cite{Mon}, p. 144) his conjecture on existence of \emph{idempotent} counterexamples.

There has been a number of attempts on the Montgomery problem. In particular, led by the examples of Hardy-Littlewood and Boas, Mockenhaupt \cite{Moc} expressed his view that $1+e_1\pm e_{k+2}$, where $2k<p<2k+2$, should provide a counterexample in the Hardy-Littlewood majorant problem, (at least for $k=1,2$). So we are to discuss the following reasonably documented conjecture.

\begin{conjecture}\label{conj:con3} Let $2k<p<2k+2$, where $k\in \NN$ arbitrary.
Then the three-term idempotent polynomial $P_k:=1+e_1+e_{k+2}$ has
smaller $p$-norm than $Q_k:=1+e_1-e_{k+2}$.
\end{conjecture}

We have proved this for $k=0,1,2$ in \cite{Krenci}.

One motivation for us was the recent paper of Bonami and R\'ev\'esz \cite{AJM}, who used suitable idempotent polynomials as the base of their construction, via Riesz kernels, of highly concentrated ones in $L^p(\TT)$ for any $p>0$. These key idempotents of Bonami and R\'ev\'esz had special properties, related closely to the Hardy-Littlewood majorant problem. For details we refer to \cite{AJM}. For the history and relevance of this closely related problem of idempotent polynomial concentration in $L^p$ see \cite{AJM, L1conc}, the detailed introduction of \cite{Krenci}, the survey paper \cite{Ash2}, and the references therein.

As already hinted by Mockenhaupt's thesis \cite{Moc}, proving that $1+e_1\pm e_{k+2}$ would be a counterexamle in the Hardy-Littlewood majorant problem may require some numerical analysis as well. However, we designed a way to accomplish this differently than suggested by Mockenhaupt, for we don't know how to get it done along the lines hinted by him. Instead, in \cite{Krenci} we used function calculus and support our analysis by numerical integration and error estimates where necessary.

These methods are getting computationally more and more involved when $k$ is getting larger. Striving for a worst-case error bound in the usual Riemann numerical integration formula forces us to consider larger and larger step numbers (smaller and smaller step sizes) in the division of the interval $[0,1/2]$, where a numerical integration is to be executed. Therefore, for $k$ getting larger, we can as well expect the step numbers increase to a numerically extraneous amount, where calculations loose liability in view of the possibly accumulating small errors of the computation of the operations and regular function values -- powers, logarithms and trigonometrical or exponential functions -- involved. Any reader would readily accept a proof, which with a certain precise error estimate refers to a numerical integration formula on say a few hundred nodes, but perhaps no reader would be fully convinced reading that a numerical tabulation and integration on several tens of thousands of function values led to the numerical result. Correspondingly, in this paper we settle with the goal of keeping any numerical integration, i.e quadrature, under the step number (or number of nodes, division number) $N=500$, that is step size $h=0.001$.

Calculation of trigonometrical and exponential functions, as well as powers and logarithms, when within the numerical stability range of these functions (that is, when the variables of taking negative powers or logarithms is well separated from zero) are done by mathematical function subroutines of usual Microsoft excel spreadsheet, which computes the mathematical functions with 15 significant digits of precision. Although we do not detail the estimates of the computational error of applying spreadsheets and functions from Microsoft Excel tables, it is clear that under this step number size our calculations are reliable well within the error bounds. For a more detailed error analysis of that sort, which similarly applies here, too, see our previous work \cite{Krenci}, in particular footnote 3 on page 141 and the discussion around formula (22).

In view of the above considerations, instead of pushing forward exactly the same numerical analysis as done in \cite{Krenci} for $k=1,2$ also for higher values of $k$, (which could have been done at least for some $k$, though), here we renew the approach and invoke a number of new features of the numerical analysis. These "tricks" will enable us to keep $N$ below 500, and thus keep the invoked numerical calculations of quadratures reliable.

First, instead of the classical and simplest numerical integration by using "brute force" Riemann sums, we apply a more involved quadrature formula \eqref{eq:quadrature}, derived from Taylor approximation, which in turn allows us to keep the step number under good control. Here instead of the most famous Simpson rule, which uses only function values, we prefer a somewhat more involved quadrature, calculating the approximate value of the integral by means of using also the values of the second derivative of the integrand. The gain is considerable even if not in order, but in the constant of the error formula.

Second, as already suggested in the conclusion of \cite{Krenci}, we apply Taylor series expansion at more points than just at the midpoint $t_0:=k+1/2$ of the $t$-interval $(k,k+1)$. This reduces the size of powers of $(t-t_0)$, from powers of $1/2$ to powers of smaller radii. The Taylor polynomial of degree 7, considered in \cite{Krenci}, had error size $2^{-8}$ due to the contribution of $|\xi-t_0|^8$ in the Lagrange remainder term, while here for $k=4$ the division of the $t$-interval to $(4,4.5)$ and $(4.5,5)$ results in $O(4^{-n})$ in the respective error contribution.

\section{Notations and a few general formula for the numerical analysis}\label{sec:formulae}

Let $k\in \NN$ be fixed. (Actually we will work with $k=3$ or $k=4$ only.) To set the framework, here we briefly sketch the general scheme of our argument, and exhibit a number of general formulae for later use in the analysis.

In the sequel we write $F_{\pm}(x):=1+ e(x)\pm e((k+2)x)$ and consider the $p^{\rm th}$ power integrals $f_{\pm}(p):=\int_0^1 |F_{\pm}(x)|^p dx$ as well as their difference
$$
\Delta(p):=f_{-}(p)-f_{+}(p):=\int_0^1 |F_{-}(x)|^p-\int_0^1 |F_{+}(x)|^p dx.
$$
Our goal is to prove Conjecture \ref{conj:con3}, that is $\Delta(p)>0$ for all $p\in(2k,2k+2)$.

Let us introduce a few further notations. We will write $t:=p/2\in[k,k+1]$ and put
\begin{align}\label{eq:Gpmdef}
G_{\pm}(x)&:=|F_{\pm}(x)|^2,\qquad
g_{\pm}(t):=\frac12 f_{\pm}(2t)= \int_0^{1/2} G_{\pm}^t(x) dx,\qquad \\ \label{eq:ddef} d(t)&:=\frac 12 \Delta(2t)=g_{-}(t)-g_+(t)=\int_0^{1/2} \left[ G_{-}^t(x)-G_{+}^t(x)\right] dx.
\end{align}
So we are to prove that $d(t)>0$ for $k<t<k+1$. First we derive that at the endpoints $d$ vanishes; and, for later use, we also compute some higher order integrals of $G_{\pm}$.

\begin{lemma}\label{l:GParseval} Let $\rho\in \NN$ with $1 \leq \rho \leq k+1$. Then we have
\begin{equation}\label{eq:Gpmadelldeveloped}
G_{\pm}^{\rho}=|F_{\pm}^{\rho}|^2 =\left| \sum_{\nu=0}^{\rho\cdot (k+2)} a_{\pm}(\nu) e_{\nu} \right|^2
\end{equation}
with
\begin{equation}\label{eq:anudef}
a_{\pm}(\nu):=(\pm 1)^\mu \binom{\rho}{\mu}\binom{\rho-\mu}{\lambda},
\end{equation}
where $\mu:=\left[\dfrac{\nu}{k+2}\right]$ and $\lambda:=\nu-\mu(k+2)$ is the reduced residue of $\nu \mod k+2$. Therefore,
\begin{equation}\label{eq:Gtorhointegral}
\int_0^{1/2} |G_{\pm}|^{\rho} =\frac12 \sum_{\nu=0}^{\rho\cdot(k+2)} \left| a_{\pm}(\nu) \right|^2. \end{equation}
In particular, $\int_0^{1/2} |G_{+}|^{\rho}=\int_0^{1/2} |G_{-}|^{\rho}$ for all $0 \leq \rho\leq k+1$ and thus $d(k)=d(k+1)=0$.
\end{lemma}
\begin{remark}\label{r:whatifrholarge} By similar calculations one can compute $a_{\pm}(\nu)$ even for higher values of $\rho$ as well. E.g. in the range $k+2\leq \rho \le 2k+3$ we have $a_{\pm}(\nu)=(\pm 1)^\mu \left\{ \binom{\rho}{\mu}\binom{\rho-\mu}{\lambda} \pm \binom{\rho}{\mu-1} \binom{\rho-\mu+1}{\lambda+(k+2)} \right\}$. That we will not use, however.
\end{remark}
\begin{proof} In the trinomial development of $(1+e_1\pm e_{k+2})^{\rho}$ the general term coming from choosing $\sigma$ times $\pm e_{k+2}$ and $\tau$ times $e_1$ (and then necessarily $\rho-\sigma-\tau$ times the constant term 1) has the form $(\pm 1)^\sigma \binom{\rho}{\sigma}\binom{\rho-\sigma}{\tau}$. This to contribute to $a_{\pm}(\nu)$ we must have $\nu=\sigma(k+2)+\tau$, a condition which forces $\nu\equiv \tau \mod k+2$. Now if $\rho \leq k+1$, we also have $0\leq \tau \leq k+1$, so the number $\tau$ of choosing $e_1$s is exactly $\lambda$, the $\mod k+2$ reduced residue of $\nu$, and consequently $\sigma=(\nu-\lambda)/(k+2)=\mu$. That results in formula \eqref{eq:anudef} for $a_{\pm}(\nu)$, while \eqref{eq:Gtorhointegral} follows by Parseval's formula. Whence the assertion is proved.
\end{proof}
%%%%%%%%%Apart from the immediate result that $d$ vanishes at the endpoints of the critical interval $[k,k+1]$, we will make further use of the above explicit computation of $\rho^{\rm th}$ power integrals of $G$. To that we need the precise values of these square sums of coefficients, which is easy to bring into a more suitable form for direct calculation. Namely we have
%%%%%%%%%\begin{align}\label{eq:Arho}
%%%%%%%%%A(\rho)& :=\sum_{{\nu=0\atop \mu:=\left[\frac{\nu}{k+2}\right] }\atop \lambda:=\nu-\mu(k+2)}^{\rho\cdot(k+2)} \binom{\rho}{\mu}^2 \binom{\rho-\mu}{\lambda}^2 = \sum_{\mu=0}^{\rho} \binom{\rho}{\mu}^2 ~\sum_{\lambda=0}^{\rho-\mu} \binom{\rho-\mu}{\lambda}^2 = \sum_{\mu=0}^{\rho} \binom{\rho}{\mu}^2 \binom{2\rho-2\mu}{\rho-\mu} \notag \\
%%%%%%%%%& =1,~ 3,~ 15,~ 93,~ 639,~ 4653,~35169  \quad {\rm for} \quad \rho=0,1,2,3,4,5,6, \quad {\rm respectively}.
%%%%%%%%%\end{align}
%%%%%%%%%\begin{corollary}\label{cor:Parseval} For all $\rho\leq (k+1)$ we have $\int_0^{1/2} G_{\pm}^{\rho}= \frac12 A(\rho)$ with the constants $A(\rho)$ in \eqref{eq:Arho}.
%%%%%%%%%%%%% \label{eq:binomsums} \label{eq:Gtorho}
%%%%%%%%%\end{corollary}
%%%%%%%%%Let us note in passing, that by means of these explicit values, even integrals of arbitrary powers of $G_{\pm}$ can be estimated.

To start the analysis of $G(x):=G_{\pm}(x)$, let us compute its $x$-derivatives. We find
\begin{align}\label{eq:GpmdefFirst}
G_{\pm}(x)&=3+2\{\cos(2\pi x)\pm \cos((2k+2)\pi x)\pm \cos((2k+4)\pi x)\},\notag\\
G_{\pm}^{'}(x)&=-4\pi \sin(2\pi x)\mp (4k+4) \pi \sin((2k+2)\pi x)\mp (4k+8) \pi \sin((2k+4)\pi x)),\\
G_{\pm}^{''}(x)&=-8\pi ^2 \cos(2\pi x)\mp 8(k+1)^2\pi^2\cos((2k+2)\pi x)\mp 8(k+2)^2\pi^2\cos((2k+4)\pi x)\},\notag
\end{align}
and in general
\begin{align*}%%%%%%%%%%%%%%%%%%%%%%%%%%%%%%%%%%%%%%%%\label{eq:GpmdefLast}
G_{\pm}^{(2m+1)}(x)&=(-4)^{m+1}\pi^{2m+1} \notag \\ & \cdot \left\{ \sin(2\pi x)\pm (k+1)^{2m+1}\sin((2k+2)\pi x) \pm (k+2)^{2m+1} \sin((2k+4)\pi x) \right\}, \notag\\
G_{\pm}^{(2m)}(x)&=2(-4)^{m}\pi^{2m} \notag \\ & \cdot \left\{ \cos(2\pi x)\pm (k+1)^{2m}\cos((2k+2)\pi x)\pm (k+2)^{2m}\cos((2k+4)\pi x)\right\}.\notag
\end{align*}
Consequently we have
\begin{align}\label{eq:Gpmadmnorm}
\|G_{\pm}\|_\infty & \leq 9=:M_0, \\ \notag \| G_{\pm}^{(m)}\|_\infty & \leq 2^{m+1} \pi^{m} \{1+(k+1)^{m}+ (k+2)^{m}\}=:M_m(k)=:M_m \qquad (m=1,2,\dots).
\end{align}

We encounter a new phenomenon, compared to \cite{Krenci}, when $k=3$, since here $G_{+}(x)$ does not have a positive lower bound: we in fact have $G_{+}(1/3)=0$. (Let us note in passing that for $G_{-}$ we have $\min_{\TT} G_{-} \approx 0.282...>1/4 $ -- but we do not use this in the following.)

For higher $x$-derivatives of the composite functions $G_{+}^t \log^j G_{+}$, needed in our analysis, vanishing of $G_{+}$ causes concerns for occurring negative powers of $G_{+}$ after differentiation, while the appearance of $\log^j G_{+}$ invoke concerns of blowing up calculations and estimates in view of "$\log 0=\infty$". The first problem we resolve by a comparison of $G_{+}$ to $G'^2_{+}$, always present in the numerator, while the second difficulty will be taken care of by using only continuous functions $v^a\log^b v$, with $a>0, b\geq 0$, of $v=G_{+}(x)$. Although all this can be avoided, when $G_{-}$ is strictly bounded away from zero, for a possibly better estimation we still calculate the same comparative estimates even for $G_{-}$. (Similarly, the idea of comparison of $G_{\pm}'^2$ and $G_{\pm}$ could be used for higher $k$ as well, whether or not the functions $G_{\pm}$ vanish.)

So we want to compare $G'$ and $\sqrt{G}=|F|$, more precisely $G'^2$ and $G$. Note that $G'= 2 |F| \cdot |F|'= 2 \sqrt{G}\cdot \left(\sqrt{G}\right)'$. Another heuristical reasoning to justify the search for a bound of $G'^2/G$, is that $G\geq 0$, hence whenever $G=0$ we necessarily have $G'=0$, and the multiplicity $m$ of any zero of $G$ being an integer (as $G$ is an entire function), we conclude $m\geq 2$: so $G'^2$ has a zero of order $2(m-1)\geq m$.

So we start the search for a bound on $G'^2_{\pm}/G_{\pm}$. To this end we write $u=\cos v$ with $v=2\pi x$ and calculate
\begin{align}\label{eq:GpmprimesquareasChebyshev}\notag
G'^2_{\pm}(x)&=(4\pi)^2 (\sin v\pm (k+1) \sin (k+1)v \pm (k+2) \sin (k+2)v)^2
\notag \\&= 16 \pi^2 (1-u^2) \left[1 \pm (k+1)U_{k}(u) \pm (k+2)U_{k+1}(u)\right]^2,
\end{align}
where $U_m(u):=\dfrac{\sin((m+1)v)}{\sin v}$ ($v:=\arccos u$) is the $m$-th Chebyshev polynomial of the second kind.

We are to compare this and
\begin{equation}\label{eq:GpmasChebyshev}\notag
G_{\pm}(x)=3+2\cos v \pm2 \cos (k+1)v \pm 2 \cos (k+2) v = 3+2u \pm 2 T_{k+1}(u) \pm 2T_{k+2}(u),
\end{equation}
where here $T_m(u)=\cos(mv)$ ($v:=\arccos u$) is the $m$th Chebyshev polynomial of the first kind.

In all, $G'^2/G$ is always an entire function of $x$, and substituting $u=\cos v= \cos  2\pi x$ we have the formula
\begin{equation}\label{eq:GprimesquareperG}
\frac{G'^2(x)}{G(x)}=\frac{16 \pi^2 (1-u^2) \left[1 \pm (k+1)U_{k}(u) \pm (k+2)U_{k+1}(u)\right]^2}{3+2u \pm 2 T_{k+1}(u) \pm 2T_{k+2}(u)}.
\end{equation}

\bigskip

In the paper \cite{Krenci} we used Riemann sums and the standard Riemann sums approximation formula
$\left| \int_0^{1/2} \Phi(\alpha) d\alpha - \frac{1}{2N}\sum_{n=1}^{N}
\Phi\left(\frac{n-1/2}{2N}\right)\right| \leq \frac{\|\Phi''\|_{\infty}}{192N^2}$,
%%%\begin{equation}\label{eq:Riemann}
%%%\left| \int_0^{1/2} \Phi(\alpha) d\alpha - \frac{1}{2N}\sum_{n=1}^{N}
%%%\Phi\left(\frac{n-1/2}{2N}\right)\right| \leq \frac{\|\Phi''\|_{\infty}}{192N^2},
%%%\end{equation}
when numerically integrating functions of the form $\Phi:=H:=G^t\log^j G$ along the $x$ values.

A new feature of the present approach is that for better approximation we now improve the numerical integration method by means of invoking a quadrature formula. This was not feasible for small $t$, as higher derivatives of the composite function $H$ lead to $G$ in the denominator: the $m^{\rm{th}}$ derivative in general results in the occurrence of $G^{t-m}$, and negative powers of $G$ bear the risk of blowing up all of our estimates. This can be remedied a little by comparison of $G'^2$ to $G$, a lucky possibility explained above. This was already utilized in \cite{Krenci} to control $2^{\rm{nd}}$ derivatives of $H$, and we'll make use of it here, too for $k=3$, when for some integrals (some occurring $H$ functions) $t$ can be as small as $t=3$, while we need to control $4^{\rm{th}}$ derivatives of $H$ in view of error terms of the quadrature formula we use. With this additional consideration the $4^{\rm{th}}$ derivatives of $H$ can always be controlled for $k\geq 3$. (For $k=0,1,2$, settled in \cite{Krenci}, this could not have been possible.)

For remaining self-contained, we deduce here the otherwise well-known quadrature formula what we want to apply. This starts with the $3^{\rm{rd}}$ order Taylor polynomial approximation (with the so-called Lagrange error term), valid for four times \hbox{continuously differentiable functions $\f$:}
$$
\f(x)=\f(x_0)+\f'(x_0)(x-x_0)+\dfrac{\f''(x_0)}{2} (x-x_0)^2 + \dfrac{\f^{'''}(x_0)}{6} (x-x_0)^3 +\dfrac{\f^{IV}(\xi_{x,x_0})}{24} (x-x_0)^4.
$$
Integrating over a symmetric interval $[x_0-q,x_0+q]$ leads to
\begin{align*}
\left| \int_{x_0-q}^{x_0+q} \f(x)dx - \left\{ \f(x_0) 2q + \f''(x_0)\dfrac{q^3}{3} \right\} \right| & = \left| \int_{x_0-q}^{x_0+q} \frac{\f^{IV}(\xi_{x,x_0})}{24} (x-x_0)^4dx \right| \\ \leq  \max_{[x_0-q,x_0+q]} |\f^{IV}(x)| & \int_{x_0-q}^{x_0+q} \frac{(x-x_0)^4}{24} dx  \leq \max_{[x_0-q,x_0+q]} |\f^{IV}(x)| \, \frac{q^5}{60}.
\end{align*}
Applying the same formula for $N$ intervals of the form $[x_n-h/2,x_n+h/2]$, where $h=(b-a)/N$ and $x_n=(n-1/2)h+a$ with $n=1,\dots,N$, we obtain
\begin{equation}\label{eq:semiquadrature}
\left| \int_a^b \f - \sum_{n=1}^N \left\{ \f(x_n) h + \f''(x_n)\dfrac{h^3}{24} \right\} \right| \leq \frac{h^5}{60 ~ 2^5} \sum_{n=1}^N \max_{|x-x_n|\leq \frac{h}{2}} |\f^{IV}(x)| \leq \frac{N  h^5}{60~ 2^5} \|\f^{IV}\|_{\infty}.
\end{equation}
This leads to the following quadrature formula.\footnote{Note the noticeably better error estimate, not in order but in constant, than one would obtain by more customary Simpson type rules. This is due to the use of second derivatives, which in our case will still be calculable, explicit formulae.}
\begin{lemma}\label{l:quadrature} Let $\f$ be a four times continuously differentiable function on $[0,1/2]$. Then we have
\begin{equation}\label{eq:quadrature}
\left| \int_0^{1/2} \f (x)dx - \sum_{n=1}^N \left\{ \f\left(\frac{2n-1}{4N}\right) \frac{1}{2N} + \f'' \left(\frac{2n-1}{4N}\right) \frac{1}{192N^3} \right\} \right| \leq \frac{\|\f^{IV}\|_{\infty}}{60~2^{10} N^4} .
\end{equation}
\end{lemma}

Let us start analyzing the functions
\begin{equation}\label{eq:Hdef}
H(x):=H_{t,j,\pm}(x):=G_{\pm}^t(x)\log^j G_{\pm}(x)\qquad (x\in [0,1/2]) \qquad (t\in [k,k+1],\, j\in \NN).
\end{equation}
To find the maximum norm of $H_{t,j,\pm}$, we in fact look for the maximum of an expression of the form $v^t |\log v|^j$, where $v=G(x)$ ranges from zero (or, if $G \ne 0$, from some positive lower bound) up to $\|G\|_\infty\leq 9$. For that, a direct calculus provides the following.

\begin{lemma}\label{l:lmvsestimate} For any $s>0$ and $m\in\NN$ the function $\alpha(v):=\alpha_{s,m}(v):= v^s| \log v|^m$ behaves on $[0,\infty)$ the following way. It is nonnegative, continuous, continuously differentiable, (apart from possibly $0$ in case $s\leq 1$), has precisely two zeroes at $0$ and $1$, and it has one single critical point $v_0=\exp(-m/s)$.
Consequently, it has exactly one local maximum point at $v_0$ where its local maximum is $\left(\frac{m}{e s}\right)^m$, furthermore, the function increases in $[0,v_0]$ and also on $[1,\infty)$, and decreases on $[v_0,1]$. Therefore for any finite interval $[a,b]\subset [0,\infty)$ we have
\begin{equation}\label{eq:alphamax}
\max_{[a,b]} \alpha(v) = \begin{cases}
\alpha(b) \qquad \qquad \qquad &\textrm{if}~ a< b\leq v_0, \\
\alpha(v_0) \qquad \qquad \qquad &\textrm{if}~ a \leq v_0 < b \leq 1,\\
\max \{\alpha(v_0),\alpha(b)\} \qquad \qquad \qquad &\textrm{if}~ a \leq v_0, ~ 1 < b ,\\
\alpha(a) \qquad \qquad \qquad &\textrm{if}~ v_0 <a<b\leq 1,\\
\max \{ \alpha(a),\alpha(b)\} \qquad \qquad \qquad &\textrm{if}~ v_0 < a < 1 < b, \\
\alpha(b) \qquad \qquad \qquad &\textrm{if}~ 1 \leq a < b.
\end{cases}
\end{equation}
In particular for $[a,b]=[0,9]$ we always have
\begin{equation}\label{eq:powerlogmax}
\al_{s,m}^{*}:=\max\limits_{[0,9]} \alpha(v) = \max\left\{\left(\frac{m}{e s}\right)^m,9^s\log^m 9 \right\}=
\begin{cases}
\left(\frac{m}{e s}\right)^m \qquad & \textrm{if} ~ m/s > \frac1{\sigma_0}, \\
9^s\log^m 9  & \textrm{if} ~ m/s \leq \frac1{\sigma_0},
\end{cases}
\end{equation}
where $\sigma_0\approx 0.126... $ is the unique root of the equation $\sigma 9^{\sigma} = 1/(e\log 9)$.
\end{lemma}

For the application of the above quadrature \eqref{eq:quadrature} we calculate (c.f. also \cite{Krenci})
\begin{align}\label{eq:Hdoubleprimegeneral}
H''(x)&:=H''_{t,j,\pm}(x) = G''(x) G^{t-1}(x) \log^{j-1} G(x) \left\{t \log G(x)+j \right\} \notag \\ & ~~+ G'^2(x) G^{t-2}(x) \log^{j-2} G(x) \left\{ t(t-1) \log^2 G(x)  + j(2t-1) \log G(x) + j(j-1) \right\}.
\end{align}
However, the error estimation in the above explained quadrature approach forces us to consider even fourth $x$-derivatives of $H=H_{t,j,\pm}$. In order to calculate
\begin{equation}\label{eq:H4binom}
H^{IV}= \sum_{m=0}^4 \binom{4}{m} (G^t)^{(m)} (\log^j G)^{(4-m)},
\end{equation}
we start with computing
\begin{eqnarray}\label{eq:Gtprime}
(G^{t})'&=&tG^{t-1}G'\notag \\
(G^{t})''  &=&t(t-1)G^{t-2}G'^2+tG^{t-1}G''\notag \\
(G^{t})''' &=&t(t-1)(t-2)G^{t-3}G'^3+3t(t-1)G^{t-2}G'G''+tG^{t-1}G'''\\
(G^{t})^{IV} &=&t(t-1)(t-2)(t-3)G^{t-4}G'^4+6t(t-1)(t-2)G^{t-3}G'^2G'' \notag \\ & +&  3t(t-1)G^{t-2}G''^2+4t(t-1)G^{t-2}G'G'''+tG^{t-1}G^{IV} \notag
\end{eqnarray}
and denoting $L:=\log G$ also
\begin{align}\label{eq:Ljprime}
(L^j)'&=\frac{jL^{j-1}G'}{G} \notag\\
(L^j)''&= \frac{G'^2}{G^2}j[(j-1)L^{j-2}-L^{j-1}]+\frac{G''}{G}jL^{j-1}\notag
\\
(L^j)'''&= \frac{G'^3}{G^3}j\big[(j-1)(j-2)L^{j-3}-3(j-1)L^{j-2} +2L^{j-1} \big] \notag \notag
\\ &\qquad + \frac{G'G''}{G^2}3j[(j-1)L^{j-2}-L^{j-1}]+ \frac{G'''}{G}jL^{j-1},
\\
(L^j)^{IV}&=\frac{G'^4}{G^4} j\left[[(j-1)(j-2)(j-3)L^{j-4} -6(j-1)(j-2)L^{j-3}+11(j-1)L^{j-2}-6L^{j-1}\right] \notag
\\ &\qquad + \frac{G'^2G''}{G^3}6j\left[(j-1)(j-2)L^{j-3} -3(j-1)L^{j-2} +2L^{j-1}\right] \notag
\\ &\qquad  +\frac{G'G'''}{G^2} 4j\left[(j-1)L^{j-2}-L^{j-1}\right] +\frac{G^{IV}}{G}jL^{j-1} + \frac{G''^2}{G^2}3j\left[(j-1)L^{j-2}-L^{j-1}\right]. \notag
\end{align}
Inserting \eqref{eq:Gtprime} and \eqref{eq:Ljprime} into \eqref{eq:H4binom} we arrive at the desired general formula for $H^{IV}_{t,j,\pm}$ as follows
\begin{align}\label{eq:HIVgeneralformula}
H^{IV}&=G^{t-4}G'^4\Big\{ j(j-1)(j-2)(j-3)L^{j-4} + [4t-6]j(j-1)(j-2)L^{j-3}
\notag \\ & \qquad +[6t^2-18t+11]j(j-1)L^{j-2} +[2t^3-9t^2+11t-3]2jL^{j-1}
+ t(t-1)(t-2)(t-3) L^j\Big\}
\notag \\& + 6\cdot G^{t-3}G'^2G''\Big\{ j(j-1)(j-2)L^{j-3}+3(t-1)j(j-1)L^{j-2}
+[3t^2-6t+2)]j L^{j-1}
\notag \\&\qquad +t(t-1)(t-2)L^j \Big\}
+ 4\cdot G^{t-2}G'G'''\Big\{ j(j-1)L^{j-2}+(2t-1)jL^{j-1}+t(t-1)L^j\Big\}
 \\&  + G^{t-1}G^{IV}\left\{jL^{j-1}+tL^j \right\}
+ 3\cdot G^{t-2}G''^2\Big\{j(j-1)L^{j-2}+(2t-1)jL^{j-1}+t(t-1)L^j \Big\} \notag .
\end{align}

At all occurrences we will need an estimate for $\|H^{IV}\|_\infty$ in order to apply it in the numerical quadrature formula. Therefore, we now start estimating the above expression. For a shorter notation we write $v:=G(x)\in [0,9]$ and $\ell:=|L|=|\log v|$. As a first step we thus find for $j=1,2,3,..., t\geq3$ the estimates
\begin{align}\label{eq:HIVgeneralestimateLARGE}
|H^{IV}(x)|&\leq v^{t-4}M_1^4\Big\{ j(j-1)(j-2)(j-3)\ell^{j-4} + [4t-6]j(j-1)(j-2)\ell^{j-3}
\notag \\ & \qquad +[6t^2-18t+11]j(j-1)\ell^{j-2} +[2t^3-9t^2+11t-3]2j\ell^{j-1}
+ t(t-1)(t-2)(t-3) \ell^j\Big\}
\notag \\& + 6\cdot v^{t-3}M_1^2 M_2 \Big\{ j(j-1)(j-2)\ell^{j-3}+3(t-1)j(j-1)\ell^{j-2}
+[3t^2-6t+2]j \ell^{j-1}
\notag \\&\qquad +t(t-1)(t-2)\ell^j \Big\}
%%%% + 4\cdot v^{t-2}M_1 M_3 \Big\{ j(j-1)\ell^{j-2}+(2t-1)j\ell^{j-1}+t(t-1)\ell^j\Big\}
+ v^{t-1} M_4 \left\{j\ell^{j-1}+t\ell^j \right\}
\\& + v^{t-2} (3M_2^2+ 4M_1 M_3) \Big\{j(j-1)\ell^{j-2}+(2t-1)j\ell^{j-1}+t(t-1)\ell^j \Big\} \notag .
\end{align}
Furthermore, to be used typically for smaller values of $v=G(x)$, that is to say only for $0\leq v \leq 3$, we can derive a different estimation whenever some constant $M^*:=M^*(k)$ is known satisfying $\|G'^2/G\|_\infty \leq M^*$. Namely, we then have
\begin{align}\label{eq:HIVgeneralestimateSMALL}
|H^{IV}(x)|&\leq v^{t-2}{M^{*}}^2\Big\{ j(j-1)(j-2)(j-3)\ell^{j-4} + [4t-6]j(j-1)(j-2)\ell^{j-3}
\notag \\ & \qquad +[6t^2-18t+11]j(j-1)\ell^{j-2} +[2t^3-9t^2+11t-3]2j\ell^{j-1}
+ t(t-1)(t-2)(t-3) \ell^j\Big\}
\notag \\& + 6\cdot v^{t-2}M^{*} M_2 \Big\{ j(j-1)(j-2)\ell^{j-3}+3(t-1)j(j-1)\ell^{j-2}
+[3t^2-6t+2)]j \ell^{j-1}
\notag \\&\qquad +t(t-1)(t-2)\ell^j \Big\}
+ 4\cdot v^{t-1.5}\sqrt{M^{*}} M_3 \Big\{ j(j-1)\ell^{j-2}+(2t-1)j\ell^{j-1}+t(t-1)\ell^j\Big\}
\notag \\&  + v^{t-1}M_4 \left\{j\ell^{j-1}+t\ell^j \right\}
+ 3\cdot v^{t-2}M_2^2\Big\{j(j-1)\ell^{j-2}+(2t-1)j\ell^{j-1}+t(t-1)\ell^j \Big\} .
\end{align}
Furthermore, estimating by means of $\Lambda:=\max(\ell,1)$ and using $\ell^{j-1}, \ell^{j-2}, \ell^{j-3}, \ell^{j-4} \leq \Lambda^j$ and also $2\sqrt{v}\leq 1+v$ we are led to
\begin{align}\label{eq:HIVgeneralSMALLfurther}
|H^{IV}(x)|&\leq v^{t-2}\La^j \Bigg( {M^{*}}^2\Big\{ j(j-1)(j-2)(j-3) + [4t-6]j(j-1)(j-2)
\notag \\ & \qquad \qquad \qquad +[6t^2-18t+11]j(j-1) +[2t^3-9t^2+11t-3]2j + t(t-1)(t-2)(t-3) \Big\}
\notag\\& \qquad + 6\cdot M^{*} M_2 \Big\{ j(j-1)(j-2)+3(t-1)j(j-1)+[3t^2-6t+2)]j +t(t-1)(t-2) \Big\}
\notag \\&\qquad + 2 \sqrt{M^{*}} M_3 \Big\{ j(j-1)+(2t-1)j+t(t-1)\Big\} + 3 M_2^2\Big\{j(j-1)+(2t-1)j+t(t-1)\Big\} \Bigg)
\notag \\&  + v^{t-1}\La^j \Bigg( 2\cdot \sqrt{M^{*}} M_3 \Big\{ j(j-1)+(2t-1)j+t(t-1)\Big\}+ M_4 \left\{j+t\right\} \Bigg).
\end{align}

\section{The proof of the $k=3$ case of Conjecture \ref{conj:con3}}\label{sec:k=3}

When $k=3$, let us start with a few concrete numerical estimates of the functions $G^{(m)}_{\pm}$ and $H^{(m)}_{\pm}$. For $k=3$ we need $U_3(u)=4u(2u^2-1)$, $U_4(u)=16u^4-12u^2+1$ and $T_4(u)=8u^4-8u^2+1$, $T_5(u)=16u^5-20u^3+5u$.
Writing these in \eqref{eq:GprimesquareperG} yields for $G_{+}$
\begin{align}\label{eq:GplusprimesquareperG}
\frac{G'^2_{+}}{G_{+}}(x)&=\frac{16 \pi^2 (1-u^2) \left[80u^4+32u^3-60u^2-16u+6\right]^2}{32u^5+16u^4-40u^3-16u^2+12u+5}\notag \\& =
\frac{16 \pi^2 (1-u^2)[40u^3-4u^2-28u+6]^2}{8u^3-4u^2-8u+5},
\end{align}
canceling the common factors of $(2u+1)^2$. Note that the denominator is now non-vanishing in the interval $[-1,1]$, as its minimum is $\approx 0.12$, attained at $\dfrac{1+\sqrt{13}}{6}\approx 0.76759...$. Thus the above rational function can be maximized numerically on the range $u\in[-1,1]$ of $u=\cos(2\pi x)$, the maximum being $\approx 3699$, so
\begin{equation}\label{eq:GplusprimecompareG}
G'^2_{+} (x) < 3700 G_{+}(x).
\end{equation}

Although $G_{-}$ does not vanish, for a possibly better estimation for small values of $G_{-}(x)$, we still work out a bound on $G'^2_{-}/G_{-}$. Again with $u=\cos v$ and $v=2\pi x$ we get from \eqref{eq:GprimesquareperG}
\begin{align}\label{eq:GminusprimesquareperG}
\max_{[-1,1]} \frac{G'^2_{-}}{G_{-}}(x)&=\max_{[-1,1]} \frac{16 \pi^2 (1-u^2) \left[80u^4+32u^3-60u^2-16u+4\right]^2}{-32u^5-16u^4+40u^3+16u^2-8u+1} \approx 3865 < 3900 .
\end{align}
Note that now the denominator does not vanish and there is no singularity to make the numerical maximization difficult. Summing up, we find
\begin{equation}\label{eq:Gpmitsprimeandquotient}
\left\| \frac{G'^2_{\pm}}{G_{\pm}}\right\|_{\infty} < M^{*}(3):=3900.
\end{equation}

The next step is, as in \cite{Krenci}, to see that $d^{(j)}(3)>0$ for the first few values of $j=1,2$.

\begin{lemma}\label{l:diffder3benpoz} We have $d'(3)>0$.
\end{lemma}
\begin{remark} A preliminary numerical calculation yields the approximate value $d'(3)\approx 0.01401...$ We don't need the concrete value, but this information suggests us the proper choice of the targeted error bound of $\delta=0.007$ below.
\end{remark}
\begin{proof} From \eqref{eq:ddef} we clearly have
\begin{align}\label{eq:djdef}
d^{(j)}(t)=g_{-}^{(j)}(t)-g^{(j)}_+(t)& =\int_0^{1/2} G_{-}^t(x)\log^{j}G_{-}(x)dx -\int_0^{1/2} G_{+}^t(x)\log^{j}G_{+}(x) dx \notag \\ & =\int_0^{1/2} H_{t,j,-}(x)dx -\int_0^{1/2} H_{t,j,+}(x)dx .
\end{align}

Now we calculate the value -- that is, these two integrals -- numerically for $t=k=3$ and $j=1$. Both integrals should be computed within the error bound $\delta:=0.007$. Invoking Lemma \ref{l:quadrature} we are left with the estimation of $\|H^{IV}_{3,1,\pm}\|_\infty$. The general formula of \eqref{eq:HIVgeneralformula} now specializes to
\begin{align}\label{eq:H314}
H^{IV}= 6\frac{G'^4}{G}+ G'^2G''(66+36L)+GG'G'''(20+24L) +GG''^2(15+18L)+G^2G^{IV}(1+3L).
\end{align}
We now estimate $|H^{IV}(x)|$ distinguishing two cases, the first being when $v:=G(x)\geq 3$. Inserting the estimates of $\|G^{(m)}\|_\infty$ from \eqref{eq:Gpmadmnorm} for $m=0,1,2,3,4$, we get from \eqref{eq:H314}
\begin{align}\label{eq:H314numeric}
|H^{IV}(x)|&\leq 6\frac{(40\pi)^4}{v}+ (40\pi)^2(8\cdot\pi^2\cdot42) (66+36\log v) +v (40\pi)(16\pi^3\cdot186)(20+24\log v)
\notag \\ &\qquad\qquad+v (8\pi^242)^2(15+18\log v)+v^2 (32\pi^4882)(1+3\log v)
%%% \notag \\ &= \pi^4 \{15360000/v+ 537600 (66+36\log v) +119040v(20+24\log v)
%%% \notag \\ & +112896v(15 +18\log v) +28224v^2(1+3\log v)\}
\notag \\ &=\pi^4 \{ 15,360,000/v+ 35,481,600+ 19,353,600\log v +4,074,240 v
\\\notag &\qquad\qquad+ 4,889,088 v \log v + 28,224 v^2 + 84,672 v^2 \log v \} ,
\end{align}
which is clearly an increasing function of $v=G(x)$ for $v\geq 2$, e.g. Therefore substituting the maximal possible value $v=9$ we obtain in this case
\begin{equation}\label{eq:H314numfinal}
|H^{IV}(x)|\leq 22,444,818,695 < 2.3 \cdot 10^{10}.
\end{equation}
For smaller values of $v=G(x)$ we estimate \eqref{eq:H314} the same way as it is done in general in \eqref{eq:HIVgeneralestimateSMALL}, with $M_m$ in \eqref{eq:Gpmadmnorm} and $M^{*}$ in \eqref{eq:Gpmitsprimeandquotient} (or, we substitute $t=k=3$ and $j=1$ in \eqref{eq:HIVgeneralestimateSMALL} and use the numerical values of $M_m$ and $M^{*}$ as said). This yields
\begin{equation}\label{eq:H314tricky}
|H^{IV}(x)|\leq \notag
1.2\cdot10^9 v +  (6.7v\log v+ 1.2v^{3/2}
+ 1.4v^{3/2}\log v)\cdot10^8 + (2.8v^2 + 8.3v^2 \log v)\cdot 10^6.
\end{equation}
(Also, we could have substituted $t=k=3$ and $j=1$ in \eqref{eq:HIVgeneralSMALLfurther} and apply \eqref{eq:Gpmadmnorm} and \eqref{eq:Gpmitsprimeandquotient} in that.) The function on the right hand side takes its maximum on $[0,3]$ at $v=3$, thus
$$
|H^{IV}(x)| \leq 7.014\cdot 10^{9} < 8\cdot 10^{9}
$$
is obtained in this second case. In all, we find $\|H^{IV}(x)\| < 2.3 \cdot 10^{10}$, hence in the numerical quadrature formula \eqref{eq:quadrature} the error is estimated by
$$
\frac{2.3 \cdot 10^{10}}{60\cdot2^{10} N^4}.
$$
To bring this down below $\delta=0.007$, we need to chose the step number $N$ as large as to have
$$
\frac{2.3 \cdot 10^{10}}{60\cdot2^{10} N^4}<\delta
\qquad
\textrm{i.e.}
\qquad
N\geq N_0:=  \sqrt[4]{\frac{2.3 \cdot 10^{10}}{60\cdot 2^{10}\cdot0.007}}\approx 86....
$$
Calculating the quadrature formula with $N=100$, we obtain the
approximate value 0.014012641..., whence  $d'(3)> 0.014012641...
-2 \cdot0.007
> 0$.
\end{proof}

\begin{lemma}\label{l:diff2ndder3benneg} We have $d''(3)>0$.
\end{lemma}
\begin{remark} Preliminary numerical calculation yields $d''(3) \approx 0.087602...$ .
\end{remark}
\begin{proof} From \eqref{eq:djdef} now we calculate the value -- that is, these two integrals -- numerically for $t=k=3$ and $j=2$. Both integrals should be computed within the error bound $\delta:=0.04$. As before, invoking Lemma \ref{l:quadrature} we are left with the estimation of $\|H^{IV}_{3,2,\pm}\|_\infty$. The general formula of \eqref{eq:HIVgeneralformula} now specializes to
\begin{align}\label{eq:H324}
H^{IV}&= G'^2G''(36L^2+132L+72)+GG''^2 (18L^2+30 L+6) \notag \\ &+GG'G'''(24L^2+40L+8) + G^2G^{IV}(3L^2+2L) + \frac{G'^4}{G}(12L+22).
\end{align}
Similarly as before, in the estimation of $|H^{IV}(x)|$ we distinguish two cases. Namely, we separate cases according to $v:=G(x)\geq e$ or $0\leq v <e$. For the case when $e \leq v\leq 9$, i.e. $1 \leq L\leq \log 9$, application of \eqref{eq:Gpmadmnorm} after substituting $t=k=3$ and $j=2$ in \eqref{eq:HIVgeneralestimateLARGE} (in other words, using \eqref{eq:Gpmadmnorm} in estimating \eqref{eq:H324}) yields
\begin{align}\label{eq:H324numeric}
|H^{IV}(x)|&\leq (40\pi)^2336\pi^2 (36L^2+132L+72)+G(336\pi^2)^2(18L^2+30 L+6)+G 40\pi 2976 \pi^3
\notag \\ &\qquad\qquad \cdot(24L^2+40L+8) + G^2 28,224 \pi^4 (3L^2+2L) + \frac{(40\pi)^4}{G}(12L+22)
\notag \\ & \leq \pi^4 \{ 288,064,517.... +43,137,255...v + 344,030...v^2+ \frac{123,818,739}{v}  \},
\end{align}
which is, by easy calculus, an increasing function of $v=G(x)$ for $v\geq 2$, e.g. Therefore substituting the maximal possible value $v=9$ yields in this case
\begin{equation}\label{eq:H324numfinal}
|H^{IV}(x)|< 7\cdot 10^{10}.
\end{equation}
%%%%%
%%%%%\newpage
%%%%%\vskip1cm
%%%%%\centerline{\bf RECONSIDER FROM HERE TO "SUMMING UP"!}
%%%%%\vskip1cm

For smaller values of $G(x)$ when $0\leq v:=G(x)\leq e$, in \eqref{eq:HIVgeneralSMALLfurther} we substitute $t=k=3$ and $j=2$ and then use \eqref{eq:Gpmadmnorm} for $m\geq 2$ and also \eqref{eq:Gpmitsprimeandquotient}, leading to
\begin{align}\label{eq:H324tricky}
|H^{IV}(x)|&\leq v \La^2 \left( 34{M^{*}}^2 + 240 M^{*} M_2 + 36 \sqrt{M^{*}} M_3
+ 54 M_2^2 \right)+ v^2\La^2 \left( 36 \sqrt{M^{*}} M_3 + 5 M_4 \right)
\notag \\& \le e \left( 34\cdot {3900}^2 + 240 \cdot 3900 \cdot 336 \pi^2
+ 36 \sqrt{3900} \cdot 3040 \pi^3 + 54 \cdot 336^2 \pi^4 \right)
\\ &  \qquad + e^2 \left( 36 \sqrt{3900} \cdot 3040\pi^3 + 5 \cdot 28,224 \pi^4 \right) \approx 13,700,830,408 < 2 \cdot 10^{10},\notag
\end{align}
applying also that on $[0,e]$ $v\La^2\leq e$ and $v^2\La^2 \leq e^2$.
%%%%%%%%%%
%%%%%%%%%%\rec{
%%%%%%%%%%Substituting $t=k=3$ and $j=2$  the same can be computed also from \eqref{eq:H324} along the way used in \eqref{eq:HIVgeneralSMALLfurther}, i.e. using $2G^{3/2}\leq G+G^2$, estimating by $\La^2$ all powers $L^0,L,L^2$ which occur, and then writing in $M_2-M_4$ and $M^{*}$, as given by \eqref{eq:Gpmadmnorm} and \eqref{eq:Gpmitsprimeandquotient}. This also gives
%%%%%%%%%%$$
%%%%%%%%%%%%%\begin{align}\label{eq:H324consequence}
%%%%%%%%%%H^{IV} \leq  \left\{ 240 M_2 M^{*} + 54 M_2^2 + 36 M_3 \sqrt{M^{*}} + 34 {M^{*}}^2 \right\} G \La^2 + \left\{36 M_3 \sqrt{M^{*}} + 5 M_4 \right\} G^2  \La^2 \leq ...
%%%%%%%%%%$$
%%%%%%%%%%%%%\end{align}
%%%%%%%%%%}

Summing up,
$$
\|H^{IV}\|_{\infty} \leq 7\cdot 10^{10},
$$
hence in the numerical quadrature formula \eqref{eq:quadrature} the error is estimated by
$$
\frac{7 \cdot 10^{10}}{60\cdot2^{10} N^4}.
$$
We need to chose the step number $N$ large enough to bring this error below $\delta=0.04$, i.e. to have
$$
N\geq N_0:=  \sqrt[4]{\frac{7 \cdot 10^{10}}{60\cdot2^{10}\cdot0.04}}\approx 73.05... .
$$
Calculating the quadrature formula with $N=100$, i.e. step size $h=0.005$, we obtain the numerical approximate value 0.08760174..., so $d''(3)> 0.08760174... -2 \cdot0.04 > 0$.
\end{proof}

Our next aim will be to show that $d''$ is concave in $[3,4]$, i.e. that $d^{IV}<0$. That will be the content of Lemma \ref{l:dthreeprime}. To arrive at it, our approach will be a computation of some approximating polynomial, which is, apart from a possible slight and well controlled error, a Taylor polynomial of $d^{IV}$.

Numerical tabulation of values gives that $d^{IV}$ is decreasing from $d^{IV}(3)\approx -0.068447...$ to even more negative values as $t$ increases from 3 to 4. Thus our goal is to set $n\in \NN$ and $\delta_j>0$, ($j=0,\dots,n+1)$ suitably so that in the Taylor expansion
\begin{equation}\label{eq:d3Taylor_2}
d^{IV}(t)=\sum_{j=0}^n \frac{d^{(j+4)}(\frac72)}{j!}\left(t-\frac72\right)^j +R_{n}(d^{IV},t),\quad
R_{n}(d^{IV},t)=\frac{d^{(n+5)}(\xi)}{(n+1)!}\left(t-\frac72\right)^{n+1}
\end{equation}
the standard error estimate
\begin{align}\label{eq:Rd5t}
|R_n(d^{IV},t)|& \leq \frac{\|H_{\xi,n+5,+}\|_{L^1[0,1/2]} + \|H_{\xi,n+5,-}\|_{L^1[0,1/2]}}{(n+1)!} \cdot 2^{-(n+1)} \notag \\ &\leq  \frac{\frac12\|H_{\xi,n+5,+}\|_\infty + \frac12\|H_{\xi,n+5,-}\|_\infty }{(n+1)! 2^{n+1}}\\ & \leq \frac{\max_{3\leq \xi\leq 4} \|H_{\xi,n+5,+}\|_\infty + \max_{3\leq \xi\leq 4} \|H_{\xi,n+5,-}\|_\infty}{(n+1)! 2^{n+2}}\notag
\end{align}
provides the appropriately small error $\|R_n(d^{IV},\cdot)\|_\infty <\delta_{n+1}$, while with appropriate approximation $\overline{d}_j$ of $d^{(j+4)}(7/2)$,
\begin{equation}\label{eq:djoverbarcriteria}
\left\|\frac{d^{(j+4)}(\frac72)-\overline{d}_j}{j!}\left(t-\frac72\right)^j\right\|_\infty =\frac{\left|d^{(j+4)}(\frac72)-\overline{d}_j\right|}{2^j j!}< \delta_j\qquad (j=0,1,\dots,n).
\end{equation}
Naturally, we wish to chose $n$ and the partial errors $\delta_j$ so that $\sum_{j=0}^{n+1}\de_j <\de:=0.068$, say, so that $d^{IV}(t)< P_n(t) +\de$ with
\begin{equation}\label{eq:Pndef}
P_n(t):=\sum_{j=0}^n \frac{\overline{d}_j}{j!}\left(t-\frac72\right)^j.
\end{equation}
Here the approximate values $\dbj$ will be obtained by numerical integration, using the quadrature formula \eqref{eq:quadrature} in approximating $d^{j+4}(7/2)$, which has the integral representation \eqref{eq:djdef} with $j=0,\dots,n$. To be precise, we apply the error formula of \eqref{eq:quadrature} with $N_j\in\NN$ steps, where $N_j$ are set in function of a prescribed error of approximation $\eta_j$, which in turn will be set in function of the choice of $\de_j$.

So now we carry out this programme. First, as $G_{\pm}(x)\in [0,9]$, $|G^{\xi}_{\pm}(x)\log^m G_{\pm}(x)|\leq \max_{[0,9]} |v^{\xi}\log^{m}v |=\alpha_{\xi,m}^{*}$, which we consider with $\xi\in[3,4]$ and $m=n+5 \geq 4$. By Lemma \ref{l:lmvsestimate} we derive for all $n\leq 18$ that
\begin{equation}\label{eq:max34}
\|H_{\xi,n+5,\pm}(x)\|_{\infty} \leq  9^4 \log^{n+5} 9 =6561\cdot 2^{n+5} \log^{n+5}3 \qquad (3 \leq \xi\leq 4,\,1 \leq n \leq 18).
\end{equation}
In view of \eqref{eq:Rd5t} this yields $|R_n(d^{(4)},t)|\leq \dfrac{104,976 \log^{n+5}3}{(n+1)!}<0.011=:\de_{11}$ for $n=10$.

Now we must set $\de_0,\dots,\de_{10}$, too. So let now $\de_j=0.005$ for each $j=0,\dots,10$. The goal is that the termwise error \eqref{eq:djoverbarcriteria} would not exceed $\de_j$, which will be guaranteed by $N_j$ step quadrature approximation of the two integrals defining $d^{(j+4)}(7/2)$ with prescribed error $\eta_j$ each. Therefore, we set $\eta_j:=\de_j 2^jj!/2$, and note that in order to have \eqref{eq:djoverbarcriteria} it suffices that
\begin{equation}\label{eq:Njchoice}
N_j > N_j^{\star}:=\sqrt[4]{\frac{\|H^{IV}_{7/2,j+4,\pm}\|_\infty}{60\cdot 2^{10} \eta_j}} = \sqrt[4]{\frac{\|H^{IV}_{7/2,j+4,\pm}\|_\infty}{60\cdot 2^{10} j! 2^{j-1} \de_j}}
\end{equation}
according to Lemma \ref{l:quadrature}. So at this point we estimate $\|H^{IV}_{7/2,j+4,\pm}\|_\infty$ for $j=0,\dots,10$ to find appropriate values of $N_j^{\star}$.

\begin{lemma}\label{l:HIVnorm} For $j=0,\dots,10$ we have the numerical estimates of Table \ref{table:HIVnormandNj} for the values of $\|H^{IV}_{7/2,j+4,\pm}\|_\infty$. Setting $\de_j=0.005$ for $j=0,\dots,10$ the approximate quadrature of order $500=:N=:N_j\geq N_j^{\star}$ with the listed values of $N_j^{\star}$ yield the approximate values $\overline{d}_j$ as listed in Table \ref{table:HIVnormandNj}, admitting the error estimates \eqref{eq:djoverbarcriteria} for $j=0,\dots,10$. Furthermore, $\|R_{10}(d^{IV},t)\|_{\infty} <0.011=:\de_{11}$ and thus with the approximate Taylor polynomial $P_{10}(t)$ defined in \eqref{eq:Pndef} the approximation $|d^{IV}(t)-P_{10}(t)|<\de:=0.068$ holds uniformly for $ t \in [3,4]$.
\begin{table}[h!]
\caption{Estimates for \label{table:HIVnormandNj} values of $\|H^{IV}_{7/2,j+4,\pm}\|_\infty$, corresponding values of $N_j^{\star}$ with $\de_j:=0.005$, and values of
$\overline{d_j}$ with $N:=N_j:=500$ for $j=0,\dots,10$.}
\begin{center}
\begin{tabular}{|c|c|c|c|}
%%%\toprule
$j$ \qquad & estimate for $\|H^{IV}_{7/2,j+4,\pm}\|_\infty$ & $N_j^{\star}$ & $\overline{d_j}$\\
%%% \midrule
0 & $ 3.3 \cdot 10^{12}$ & 383 & -8.097236891\\
1 & $ 9.1 \cdot 10^{12}$ & 415 & -37.59530251\\
2 & $ 2.5 \cdot 10^{13}$ & 378 & -141.3912224\\
3 & $ 6.8 \cdot 10^{13}$ & 310 & -468.2134571\\
4 & $ 1.9 \cdot 10^{14}$ & 239 & -1423.831595\\
5 & $ 4.8 \cdot 10^{14}$ & 169 & -4074.963995\\
6 & $ 2.8 \cdot 10^{15}$ & 142 & -11,148.7318\\
7 & $ 2.6 \cdot 10^{16}$ & 128 & -29,465.89339\\
8 & $ 2.7 \cdot 10^{17}$ & 115 & -75,792.43387\\
9 & $ 2.9 \cdot 10^{18}$ & 101 & -190,751.6522\\
10 & $ 3.4 \cdot 10^{19}$ & 88 & -471,634.7482\\
%% \bottomrule
\end{tabular}
\end{center}
\end{table}
\end{lemma}
\begin{proof} We start with the numerical upper estimation of $H^{IV}_{7/2,j,\pm}(x)$ for $3\leq x\leq 4$, where now in view of the shift of indices we need the estimation for $4\leq j \leq 14$. All what follows is not sensitive to $j\leq 14$, but it is convenient that $j\geq 4$, as otherwise in some derivatives the powers of $L(x)=\log G(x)$ would diminish, changing the formula slightly.

When $v=G(x)\geq e$, substitution of $t=7/2$ in \eqref{eq:HIVgeneralestimateLARGE} while using the estimates \eqref{eq:Gpmadmnorm} of $\|G^{(m)}\|$ and $\ell\leq \log 9< 2.2$ yields the estimate
\begin{align*}
|H^{IV}_{7/2,j,\pm}(x)|
\leq  2.2^j\Big\{\frac{A}{\sqrt{G}}+ B \cdot \sqrt{G}\Big\},
\end{align*}
where
\begin{align*}
A:=A_j:&=9\cdot 10^6 \left[j^4-2j^3+105j^2+134j+154 \right] \\
B:=B_j:&= 2.8 \cdot 10^7 j^3 +3.9\cdot 10^8 j^2 + 5.7 \cdot 10^{9} j + 1.2 \cdot 10^{10}.
\end{align*}
This last function is a strictly convex function of $u:=\sqrt{v}=\sqrt{G(x)}$
-- so it must have a unique minimum and two monotonic parts before
and after the minimum point. Easy calculus yields that the minimum
is located at $v_0:=\frac{A}{B}$. Now $v_0$ is less than $e$, when
$j<14$, hence then the function is increasing on $[e,9]$ and it
achieves its maximum at $v=9$. When $j=14$, the minimum falls
inside $[e,9]$, and the maximum is the maximum of the values at
$e$ and $9$, but the latter being much larger, we again find that
our estimate is maximized taking $v=9$. Finally, substitution of $v=9$ yields
\begin{equation}\label{eq:HIVgenerallargefinal}
|H^{IV}_{7/2,j,\pm}(x)| \leq 2.2^j \cdot \left\{3\cdot 10^6 j^4 + 7.8\cdot 10^7 j^3 + 1.5 \cdot 10^9 j^2 + 1.8 \cdot 10^{10} j + 3.7 \cdot 10^{10} \right\}.
\end{equation}

When $G(x)<e$, we substitute $k=3$ and $t=7/2$ in \eqref{eq:HIVgeneralSMALLfurther}, and apart from the values of $M_2,M_3$ and $M_4$ we also use the last estimate of \eqref{eq:Gpmitsprimeandquotient}. Further, we write in $v^{5/2}\La^j \leq e\cdot v^{3/2}\La^j$ and, by means of Lemma \ref{l:lmvsestimate}, $\max_{[0,e]} v^{3/2} \La^j= \max \left( e^{3/2}, (\frac{2j}{3e})^j \right)$
%%%% $=\alpha^{*}_{3/2,j}$
which then yields
\begin{align}\label{eq:HIVgenSMALLfurtherk3t72}
|H^{IV}_{7/2,j,\pm}(x)|&\leq
\max\left( e^{3/2},\left( \frac{2j}{3e}\right)^j \right)
%%%% \alpha^{*}_{3/2,j}
\left\{ 1.6 \cdot 10^{7} j^4 + 1.1 \cdot 10^{8} j^3 +  5.6 \cdot 10^{8} j^2 + 1.6 \cdot 10^{9} j + 1.9 \cdot10^{9} \right\}.
\end{align}

From here we take the maximum of \eqref{eq:HIVgenerallargefinal} and the above \eqref{eq:HIVgenSMALLfurtherk3t72} for all $j=6,\dots,14$, which means using \eqref{eq:HIVgenerallargefinal} up to $j=9$ and then \eqref{eq:HIVgenSMALLfurtherk3t72} for $j=10,\dots,14$, leading to the upper estimates of $\|H^{IV}_{7/2,j,\pm}\|_{\infty}$ as listed in Table \ref{table:HIVnormandNj}.

Finally, we collect also the resulting numerical estimates of $N_j^{\star}$ -- as given by the formulae \eqref{eq:Njchoice} -- in Table \ref{table:HIVnormandNj} and furthermore list the accordingly computed values of $\overline{d_j}$, too, applying the numerical quadrature formula \eqref{eq:quadrature} with step size $h=0.001$, i.e. $N=N_j=500$ steps.
\end{proof}

\begin{lemma}\label{l:dthreeprime} We have $d^{IV}(t)<0$ for all $3 \leq t \leq 4$.
\end{lemma}

\begin{proof} We approximate $d^{IV}(t)$ by the polynomial $P_{10}(t)$ constructed in \eqref{eq:Pndef} as the approximate value of the order 10 Taylor polynomial of $d^{IV}$ around $t_0:=7/2$. As the error is at most $\de$, it suffices to show that $p(t):=P_{10}(t)+\de<0$ in $[3,4]$.

Now $P_{10}(3)=-0.068458667...$ so $P_{10}(3)+\de<0$.

Moreover, $p'(t)=P_{10}'(t)=\sum_{j=1}^{10} \dfrac{\overline{d}_j}{(j-1)!} (t-7/2)^{j-1}$ and $p'(3)=-4.00969183<0$. From the explicit formula of $p(t)$ we consecutively compute also $p''(3)=-23.12291565<0$, $p'''(3)=-93.80789264<0$ and $p^{(4)}(3)=-324.0046433$, $p^{(5)}(3)=-978.7532737...$, $p^{(6)}(3)=-3144.062078...$,  $p^{(7)}(3)=-5587.909055...$, all $<0$. Thus $p^{(j)}(3)<0$ for $j=0,\dots,7$.

Therefore in order to conclude $p(t)<0$ for $3\leq t\leq 4$ it suffices to show that $p^{(8)}(t)=\overline{d}_8+\overline{d}_9(t-7/2)+ (\overline{d}_{10}/2) (t-7/2)^{2}$ stays negative in the interval $[3,4]$. However, the leading coefficient of $p^{(8)}$ is negative, while it is easy to see that the discriminant $\Delta:=\overline{d}_9^2-2\overline{d}_8 \overline{d}_{10}$ of $p^{(8)}$ is negative, too: $\Delta\approx -3.511\cdot 10^{10}$. Therefore, the whole parabola of the graph of $p^{(8)}$ lies below the $x$-axis i.e. $p^{(8)}(t)<0$ ($\forall t \in \RR$). It follows that also $p(t)<0$ for all $t\geq 3$.
\end{proof}

\begin{proof}[Proof of the $k=3$ case of Conjecture \ref{conj:con3}]

Since $d(3)=d(4)=0$, and $d'(3)>0$, $d$ takes some positive values close to $3$; so in view of Lagrange's (Rolle's) theorem, $d'$ takes some negative values
as well. Therefore, $d'$ decreases from a positive value at $3$ to
some negative value somewhere later; it follows that $d''$ takes
some negative values in $(3,4)$. Also, $d''$ is concave and
$d''(3)>0$ implies that $d''$ changes from positive values towards
negative ones; by concavity, there is a unique zero point $\tau$
of $d''$ in $(3,4)$, where $d''$ has a definite sign change from
positive to negative.

It follows that $d'$, starting with the positive value at 3, first
increases, achieves a maximal positive value at $\tau$, and then
it decreases, reaches zero and then eventually negative values, as
seen above. That is, when it becomes zero at some point $\sigma$,
it already has a negative derivative, and it keeps decreasing from
that point on. So $d'$ is positive until $\sigma$, when it has a
strict sign change and becomes negative until 4. Therefore, $d$
increases until $\sigma$ and then decreases till 4; so $d$ forms a
cap shape and it is minimal at the endpoints $3$ and $4$, where it
vanishes. It follows that $d>0$ in $(3,4)$.

This concludes the proof of the $k=3$ case of Conjecture \ref{conj:con3}.
\end{proof}

\section{The case $k=4$ of Conjecture \ref{conj:con3}}\label{sec:final}

First of all let us record that in case $k=4$ in \eqref{eq:GpmasChebyshev} we are to deal with $G_{\pm}(x)=3+2u\pm2[T_5(u)+T_6(u)]\qquad (u=\cos2\pi x)$
so putting in $T_5(u)= 16u^5-20u^3+5u$ and $T_6(u)= 32u^6-48u^4+17u^2$ a numerical calculation of the occurring polynomials give
\begin{equation}\label{eq:minGpmk4}
\min_{\TT} G_{+} \approx 0.0946...  \qquad {\textrm{and}}\qquad \min_{\TT} G_{-} \approx 0.02776...
\end{equation}
Therefore in case $k=4$  we can estimate $\ell:=|L|=|\log G(x)|< |\log(0.027)|<3.7$ for both signs of $G_{\pm}$.

\begin{lemma}\label{l:diffder4benpoz} We have $d'(4)>0$.
\end{lemma}
\begin{remark} By numerical calculation, $d'(4)\approx 0.0062067...$.
\end{remark}

\begin{proof}
From \eqref{eq:HIVgeneralformula} with $t=4, j=1$
\begin{align}\label{eq:H414}
 H^{IV}&= G'^4(50+24L)+ 6GG'^2G''(26+24L)+4G^2G'G'''(7+12L)+G^3G^{IV}(1+4L)
\\ & \notag +3G^2G''^2(7+12L).
\end{align}
From this, $\ell<3.7$ and the estimates \eqref{eq:Gpmadmnorm} we get by plain substitution as before $\|H^{IV}\|_\infty < 1.6 \cdot 10^{12}$.

To bring the resulting error estimate down below $\delta=0.003$, we need to chose the step number $N$ as large as to have
$$
\frac{1.6 \cdot 10^{12}}{60\cdot2^{10} N^4}<\delta
\qquad
\textrm{i.e.}
\qquad
N\geq N_0:=  \sqrt[4]{\frac{1.6 \cdot 10^{12}}{60\cdot 2^{10}\cdot0.003}}\approx 306....
$$
Calculating the quadrature formula with $N=500$, we obtain the
approximate value 0.0062067..., whence  $d'(4)> 0.0062067...
-2 \cdot0.003> 0$.
\end{proof}

\begin{lemma}\label{l:diff2ndder4benneg} We have $d''(4)>0$.
\end{lemma}
\begin{remark} By numerical calculation, $d''(4)\approx 0.0541341...$.
\end{remark}
\begin{proof}
From \eqref{eq:HIVgeneralformula} with $t=4, j=2$ and inserting the values of $M_m$s given by (\ref{eq:Gpmadmnorm}) together with $\ell<3.7$ we get $\|H^{IV}\|_\infty < 7\cdot 10^{12}$.

Thus to bring the error below $\delta=0.027$, we need to chose the step number $N$ large enough to have
$$
\frac{7 \cdot 10^{12}}{60\cdot2^{10} N^4}<\delta
\qquad
\textrm{i.e.}
\qquad
N\geq N_0:=  \sqrt[4]{\frac{7 \cdot 10^{12}}{60\cdot 2^{10}\cdot0.027}}\approx 255....
$$
Calculating the quadrature formula with $N=500$, we obtain the approximate value 0.05413417..., whence  $d''(4)> 0.05413417... -2 \cdot0.0027> 0$.
\end{proof}

\begin{lemma}\label{l:diff3ndder4benneg} We have $d'''(4)>0$.
\end{lemma}
\begin{remark} By numerical calculation, $d'''(4)\approx 0.2255707...$.
\end{remark}
\begin{proof}
From \eqref{eq:HIVgeneralformula} with $t=4, j=3$ and calculating with the same values of $M_m$ from \eqref{eq:Gpmadmnorm} as above -- together with $\ell <3.7$ in view of \eqref{eq:minGpmk4} -- we arrive at
\begin{align*}
H^{IV}&= G'^4(60+210L+150L^2+24L^3)+ 6GG'^2G''(6+54L+78L^2+24L^3) \notag
\\ & +4G^2G'G'''(6L+21L^2+12L^3)+G^3G^{IV}(3L^2+4L^3)+3G^2G''^2(6L+21L^2+12L^3),
\\ \notag |H^{IV}|& <5.4\cdot 10^{13}.
\end{align*}
To bring this down below $\delta=0.112$, we need to chose the step number $N$ large enough to have
$$
\frac{5.3 \cdot 10^{13}}{60\cdot2^{10} N^4}<\delta
\qquad
\textrm{i.e.}
\qquad
N\geq N_0:=  \sqrt[4]{\frac{5.3 \cdot 10^{13}}{60\cdot 2^{10}\cdot0.112}}\approx 298....
$$
Calculating the quadrature formula with $N=500$, we obtain the
approximate value 0.22557089..., whence  $d'''(4)> 0.22557089...
-2 \cdot0.112> 0$.
\end{proof}

So we arrive at the analysis of $d^{V}$. Numerical tabulation of values give that $d^{V}$ is decreasing from $d^{V}(4)\approx -2,217868...$ to even more negative values as $t$ increases from 4 to 5. So we now set forth \emph{proving} that $d^{V}<0$ in $[4,5]$. To arrive at it, our approach will be a computation of some approximating polynomial $p(t)$, which is, within a small and well controlled error, will be a Taylor polynomial of $d^{V}(t)$. However, as we intend to keep the step number $N$ of the numerical integration under 500, we take the liberty of approximating $d^{V}$ by different polynomials (using different Taylor expansions) on various subintervals of $[4,5]$. More precisely, we divide the interval $[4,5]$ into 2 parts, and construct approximating Taylor polynomials around $4.25$ and $4.75$.

So now setting $t_0=4.25$ or $t_0=4.75$, the Taylor approximation will have the form
\begin{equation}\label{eq:dVTaylor45}
d^{V}(t)=\sum_{j=0}^n \frac{d^{(j+5)}(t_0)}{j!}\left(t-t_0\right)^j +R_{n}(d^{V},t_0,t),\quad
R_{n}(d^{V},t_0,t)=\frac{d^{(n+6)}(\xi)}{(n+1)!}\left(t-t_0\right)^{n+1}.
\end{equation}
Therefore instead of \eqref{eq:Rd5t} we can use
\begin{align}\label{eq:Rd5tnegyed}
|R_n(d^{V},t_0,t)|& \leq \frac{\|H_{\xi,n+6,+}\|_{L^1[0,1/2]} + \|H_{\xi,n+6,-}\|_{L^1[0,1/2]}}{(n+1)!} \cdot 4^{-(n+1)} \notag \\ &\leq  \frac{\frac12\|H_{\xi,n+6,+}\|_\infty + \frac12\|H_{\xi,n+6,-}\|_\infty }{(n+1)! 2^{2n+2}}\\ & \leq \frac{\max_{|\xi-t_0|\leq 1/4} \|H_{\xi,n+6,+}\|_\infty + \max_{|\xi-t_0|\leq 1/4} \|H_{\xi,n+6,-}\|_\infty}{(n+1)! 2^{2n+3}}.\notag
\end{align}
So once again we need to maximize \eqref{eq:Hdef}, that is functions of the type $v^{\xi}|\log v|^m $, on $[0,9]$ (or, more precisely, on the subinterval ${\mathcal R}(G)\approx [0.02776...,9]$, where the values are actually attained by $v:=G(x)$).
So now similarly to \eqref{eq:max34}, we get from \eqref{eq:powerlogmax} of Lemma \ref{l:lmvsestimate} that for any $m \leq 31$ and $|\xi- t_0|\leq 1/4$
\begin{equation*}\label{eq:max45}
\|H_{\xi,m,\pm}\|_\infty = \max\left\{\left(\frac{m}{e\cdot\xi}\right)^m,9^{\xi}\log^m 9 \right\} =9^\xi\log^m9\leq 9^{t_0+1/4}\log^m 9 .
\end{equation*}
In all, for $n\leq 25$
\begin{equation}\label{eq:maxmax45}
\max_{|\xi-t_0|\leq 1/4} \|H_{\xi,n+6,\pm}(x)\|_{\infty} \leq \begin{cases} 9^{4.5} \log^{n+6} 9 = 19,683\, 2^{n+6} \log^{n+6}3 \,~ &\textrm{if}\,t_0=4.25,
\\ 9^{5}\log^{n+6} 9 = 59,049\, 2^{n+6} \log^{n+6}3 & \textrm{if}\,t_0=4.75.
\end{cases}
\end{equation}
In case $t_0=4.25$ now we chose $n=7$. Then for this case the Lagrange remainder term \eqref{eq:Rd5tnegyed} of the Taylor formula \eqref{eq:dVTaylor45} can be estimated as
$|R_n(d^{V},t)|\leq \dfrac{314,928 \log^{n+6}3}{2^n(n+1)!}<0.21=:\de_{8}$.

As before, the Taylor coefficients $d_{j+5}(t_0)$ cannot be obtained exactly, but only with some error, due to the necessity of some kind of numerical integration in the computation of the formula \eqref{eq:djdef}. Hence we must set the partial errors $\de_0,\dots,\de_{7}$ with $\sum_{j=0}^{8}\de_j <\de:=2.21$, say, so that $d^{V}(t)< P_n(t) +\de$ for
\begin{equation}\label{eq:Pndef425}
P_n(t):=\sum_{j=0}^n \frac{\overline{d}_j}{j!}\left(t-4.25\right)^j.
\end{equation}
The analogous criteria to \eqref{eq:djoverbarcriteria} now has the form:
\begin{equation}\label{eq:djoverbarcriteria425}
\left\|\frac{d^{(j+5)}(4.25)-\overline{d}_j}{j!}\left(t-4.25\right)^j\right\|_\infty =\frac{\left|d^{(j+5)}(4.25)-\overline{d}_j\right|}{2^{2j} j!}< \delta_j\qquad (j=0,1,\dots,n).
\end{equation}
That the termwise error \eqref{eq:djoverbarcriteria425} would not exceed $\de_j$ will be guaranteed by $N_j$ step quadrature approximation of the two integrals in \eqref{eq:djdef} defining $d^{(j+5)}(4.25)$ with prescribed error $\eta_j$ each. Therefore, we set $\eta_j:=\de_j 2^{2j}j!/2$, and note that in order to have \eqref{eq:djoverbarcriteria425} \begin{equation}\label{eq:Njchoice425}
N_j > N_j^{\star}:=\sqrt[4]{\frac{\|H^{IV}_{4.25,j+5,\pm}\|_\infty}{60\cdot 2^{10} \eta_j}} = \sqrt[4]{\frac{\|H^{IV}_{4.25,j+5,\pm}\|_\infty}{60\cdot 2^{10} j! 2^{2j-1} \de_j}}
\end{equation}
suffices by the integral formula \eqref{eq:quadrature} and Lemma \ref{l:quadrature}. That is, we must estimate $\|H^{IV}_{4.25,j+5,\pm}\|_\infty$ for $j=0,\dots,7$ and thus find appropriate values of $N_j^{\star}$.

\begin{lemma}\label{l:HIVnorm425} For $j=0,\dots,7$ we have the numerical estimates of Table \ref{table:HIVnormandNj425} for the values of $\|H^{IV}_{4.25,j+5,\pm}\|_\infty$. Setting $\de_j$ as seen in the table for $j=0,\dots,7$, the approximate quadrature of order $500:=N_j\geq N_j^{\star}$ with the listed values of $N_j^{\star}$ yield the approximate values $\overline{d}_j$ as listed in Table \ref{table:HIVnormandNj425}, admitting the error estimates \eqref{eq:djoverbarcriteria425} for $j=0,\dots,7$. Furthermore, $\|R_{8}(d^{V},t)\|_{\infty} <0.21=:\de_{8}$ and thus with the approximate Taylor polynomial $P_{7}(t)$ defined in \eqref{eq:Pndef425} the approximation $|d^{V}(t)-P_{7}(t)|<\de:=2.21$ holds uniformly for $ t \in [4,4.5]$.
\begin{table}[h!]
\caption{Estimates for values of $\|H^{IV}_{4.25,j+5,\pm}\|_\infty$, chosen values of $\de_j$ and resulting $N_j^{\star}$, and $\overline{d_j}$ with $N_j:=N:=500$ for $j=0,\dots,7$.}
\label{table:HIVnormandNj425}
\begin{center}
\begin{tabular}{|c|c|c|c|c|}
%%%\toprule
$j$ \qquad & estimate for $\|H^{IV}_{4.25,j+5,\pm}\|_\infty$ & $\de_j$ & $N_j^{\star}$ & $\overline{d_j}$\\
%%% \midrule
0 & $ 1.23 \cdot 10^{15}$ & 0.65 & 499 & -11.99030682\\
1 & $ 5.32 \cdot 10^{15}$ & 0.73 & 494 & -64.72801527\\
2 & $ 2.29 \cdot 10^{16}$ & 0.4 & 492 & -273.5687453\\
3 & $ 9.80 \cdot 10^{16}$ & 0.15 & 486 & -1000.494741\\
4 & $ 4.18 \cdot 10^{17}$ & 0.04 & 486 & -3319.462864\\
5 & $ 1.77 \cdot 10^{18}$ & 0.01 & 466 & -10,266.25853\\
6 & $ 7.47 \cdot 10^{18}$ & 0.01 & 302 & -30,113.02268\\
7 & $ 3.14 \cdot 10^{19}$ & 0.01 & 188 & -84,761.00164\\

%% \bottomrule
\end{tabular}
\end{center}
\end{table}
\end{lemma}\bigskip
\begin{proof} We start with the numerical upper estimation of $H^{IV}_{4.25,j,\pm}(x)$ for $4\leq x\leq 4.5$. In the general formula \eqref{eq:HIVgeneralestimateLARGE} now we consider the case $t=4.25$ and use the estimates \eqref{eq:Gpmadmnorm} of $M_1$, $M_2$, $M_3$ and $M_4$, together with $\ell<3.7$ -- c.f. \eqref{eq:minGpmk4} --  to compute %%%\rev{Check, modify!}
\begin{align*}
|H^{IV}_{4.25,j,\pm}(x)|%%%%%&\leq%%%%%
%%%%%3.7^j\Big\{895622648.8231\dots  [0.005335721\dots j(j-1)(j-2)(j-3)+ 0.21716384\dots j(j-1)(j-2)
%%%%%\notag \\ & \qquad +3.131848064\dots j(j-1) +18.76689189\dots j
%%%%%+ 38.84765625\dots ]
%%%%%\notag \\& + 10411613292.5688\dots [ 0.019742167\dots j(j-1)(j-2)+0.712198685\dots j(j-1)
%%%%%+8.293918919\dots j
%%%%%\notag \\&\qquad +31.078125\dots ]
%%%%%+ 14357950588.9457\dots  [0.073046019\dots j(j-1)+2.027027027\dots j+13.8125\dots ]
%%%%% \\&  + 7564687782.8820\dots [0.27027027\dots j+4.25]
%%%%%\\& + 10086250377.1760\dots [0.073046019\dots j(j-1)+2.027027027\dots j+13.8125\dots ]\Big\}
%%%%%\\ & = 3.7^j\Big\{4778792.477\dots j(j-1)(j-2)(j-3) + 400044665.3\dots j(j-1)(j-2) + \\& 12005642925
%%%%%\dots j(j-1) + 1.54755 \dots  10^{11}j + 7.28152 \dots  10^{11} \Big\}
%%%%%\\&=
%%%%%3.7^j\Big\{4778792.5\dots j^4 + 371371910.4\dots j^3 +  10858075646\dots j^2
%%%%%\\& \qquad\qquad+ 143520469765.2\dots j +  728151709071.1\dots \Big\}
< 3.7^j\Big\{ 4.78\cdot10^6 j^4 + 3.72 \cdot10^8 j^3 +  1.09 \cdot 10^{10} j^2 + 1.44 \cdot 10^{11} j +  7.29 \cdot 10^{11} \Big\}.
\end{align*}
Finally, we collect the resulting numerical estimates of $\|H^{IV}\|$ in Table \ref{table:HIVnormandNj425} and list the corresponding values of $N_j^{\star}$ and $\overline{d_j}$, too, as given by the formulae \eqref{eq:Njchoice425} and the numerical quadrature formula \eqref{eq:quadrature} with step size $h=0.001$, i.e. $N=N_j=500$ steps.\end{proof}

\begin{lemma}\label{l:dthreeprime425} We have $d^{V}(t)<0$ for all $4 \leq t \leq 4.5$.
\end{lemma}
\begin{proof} We approximate $d^{V}(t)$ by the polynomial $P_{7}(t)$ constructed in \eqref{eq:Pndef425} as the approximate value of the order 7 Taylor polynomial of $d^{V}$ around $t_0:=4.25$. As the error is at most $\de$, it suffices to show that $p(t):=P_{7}(t)+\de<0$ in $[4,4.5]$. Now $P_{7}(4)=-2.2178666857...$ so $P_{7}(4)+\de<0$. Moreover, $p'(t)=P_{7}'(t)=\sum_{j=1}^{7} \dfrac{\overline{d}_j}{(j-1)!} (t-4.25)^{j-1}$ and $p'(4)=-20.41147631...<0$. From the explicit formula of $p(t)$ we consecutively compute also $p''(4)=-104.6546745...<0$, $p'''(4)=-426.8260106...<0$, $p^{(4)}(4)=-1473.198415...<0$, $p^{(5)}(4)=-5386.784165...<0$ and  $p^{(6)}(4)=-8922.772271...<0$. Finally, we arrive at $p^{(7)}(t)=\overline{d}_7$=-84,761.00164... . We have already checked that $p^{(j)}(4)<0$ for $j=0\dots 6$, so in order to conclude $p(t)<0$ for $4\leq t\leq 4.5$ it suffices to show $p^{(7)}(t)<0$ in the given interval. However, $p^{(7)}$ is constant $\overline{d}_7$, hence $p^{(7)}(t)<0$ for all $t \in \RR$.  It follows that also $p(t)<0$ for all $t\geq 4$.
\end{proof}
In case of $t_0=4.75$  we have for all $\xi\in [4.5,5]$
\begin{equation}\label{eq:max45_2}
\max\left\{\left(\frac{m}{e\cdot\xi}\right)^m,9^{\xi}\log^m 9 \right\}\leq\max\left\{\left(\frac{m}{4.5e}\right)^m,9^{5}2^m\log^m 3\right\}=9^{5}\log^m 9 \quad (\forall  m< 37).
\end{equation}
In all, $\|H_{\xi,n+6,\pm}(x)\|_{\infty} \leq 59,049\, 2^{n+6} \log^{n+6}3$ for all $4.5 \leq \xi\leq 5$ and $4\leq n < 31$. In view of \eqref{eq:Rd5tnegyed} this yields $|R_n(d^{V},t)|\leq \dfrac{944,784 \log^{n+6}3}{2^n(n+1)!}<9.1=:\de_{7}$ for $n=6$.

Next we set $\de_0,\dots,\de_{6}$. Now the criteria \eqref{eq:djoverbarcriteria} is modified as
\begin{equation}\label{eq:djoverbarcriteria475}
\left\|\frac{d^{(j+5)}(4.75)-\overline{d}_j}{j!}\left(t-4.75\right)^j\right\|_\infty =\frac{\left|d^{(j+5)}(4.75)-\overline{d}_j\right|}{2^{2j} j!}< \delta_j\qquad (j=0,1,\dots,n).
\end{equation}
Since the numerical calculation gives that $d^{V}(4.5)\approx -39.96194643...$, now we wish to chose the partial errors $\delta_j$ so that $\sum_{j=0}^{n+1}\de_j <\de:=39.9$, say, so that $d^{V}(t)< P_n(t) +\de$ with
\begin{equation}\label{eq:Pndef475}
P_n(t):=\sum_{j=0}^n \frac{\overline{d}_j}{j!}\left(t-4.75\right)^j.
\end{equation}
The goal is that the termwise error \eqref{eq:djoverbarcriteria475} would not exceed $\de_j$, which will be guaranteed by $N_j$ step quadrature approximation of the two integrals defining $d^{(j+5)}(4.75)$ with prescribed error $\eta_j$ each. Therefore, we set $\eta_j:=\de_j 2^{2j}j!/2$, and note that in order to have \eqref{eq:djoverbarcriteria475} \begin{equation}\label{eq:Njchoice475}
N_j > N_j^{\star}:=\sqrt[4]{\frac{\|H^{IV}_{4.75,j+5,\pm}\|_\infty}{60\cdot 2^{10} \eta_j}} = \sqrt[4]{\frac{\|H^{IV}_{4.75,j+5,\pm}\|_\infty}{60\cdot 2^{10} j! 2^{2j-1} \de_j}}
\end{equation}
suffices by the integral formula \eqref{eq:quadrature} and Lemma \ref{l:quadrature}. That is, we must estimate $\|H^{IV}_{4.75,j+5,\pm}\|_\infty$ for $j=0,\dots,7$ and thus find appropriate values of $N_j^{\star}$.

\begin{lemma}\label{l:HIVnorm475} For $j=0,\dots,6$ we have the numerical estimates of Table \ref{table:HIVnormandNj475} for the values of $\|H^{IV}_{4.75,j,\pm}\|_\infty$. Setting $\de_j$ as can be seen in the table for $j=0,\dots,6$, the approximate quadrature of order $500:=N_j\geq N_j^{\star}$ with the listed values of $N_j^{\star}$ yield the approximate values $\overline{d}_j$ as listed in Table \ref{table:HIVnormandNj475}, admitting the error estimates \eqref{eq:djoverbarcriteria475} for $j=0,\dots,6$. Furthermore, $\|R_{7}(d^{V},t)\|_{\infty} <9.1=:\de_{7}$ and thus with the approximate Taylor polynomial $P_{6}(t)$ defined in \eqref{eq:Pndef475} the approximation $|d^{V}(t)-P_{6}(t)|<\de:=39.9$ holds uniformly for $ t \in [4.5,5]$.
\begin{table}[h!]
\caption{Estimates for values of $\|H^{IV}_{4.75,j+5,\pm}\|_\infty$, set values of $\de_j$ and the resulting $N_j^{\star}$, and the values of $\overline{d_j}$ with $N:=N_j:=500$ steps for $j=0,\dots,6$.}
\label{table:HIVnormandNj475}
\begin{center}
\begin{tabular}{|c|c|c|c|c|}
%%%\toprule
$j$ \qquad & estimate for $\|H^{IV}_{4.75,j+5,\pm}\|_\infty$ & $\de_j$ & $N_j^{\star}$ & $\overline{d_j}$\\
%%% \midrule
0 & $ 4.98 \cdot 10^{15}$ & 8 & 378 & -111.5230149\\
1 & $ 2.13 \cdot 10^{16}$ & 9 & 373 & -432.5730847\\
2 & $ 9.07 \cdot 10^{16}$ & 7 & 339 & -1509.259877\\
3 & $ 3.85 \cdot 10^{17}$ & 3 & 323 & -4867.920658\\
4 & $ 1.63 \cdot 10^{18}$ & 1 & 305 & -14,785.12009\\
5 & $ 6.83 \cdot 10^{18}$ & 1 & 207 & -42,842.09045\\
6 & $ 2.86 \cdot 10^{19}$ & 1 & 134 & -119,563.5221\\
%% \bottomrule
\end{tabular}
\end{center}
\end{table}
\end{lemma}
\begin{proof} We start with the numerical upper estimation of $H^{IV}_{4.75,j,\pm}(x)$ for $4.5\leq x\leq 5$. In \eqref{eq:HIVgeneralestimateLARGE} now we insert $t=4.75$, use again the estimates \eqref{eq:Gpmadmnorm} of $M_1-M_4$ and $\ell<3.7$ and arrive at
\begin{align*}
|H^{IV}_{4.75,j,\pm}(x)|
< 3.7^j\Big\{ 1.44\cdot10^7 j^4 + 1.23 \cdot10^9 j^3 +  3.93 \cdot 10^{10} j^2 + 5.7 \cdot 10^{11} j +  3.18 \cdot 10^{12} \Big\}.\end{align*}
Finally, we collect the resulting numerical estimates of $\|H^{IV}\|$ in Table \ref{table:HIVnormandNj475} and list the corresponding values of $N_j^{\star}$ and $\overline{d_j}$, too, as given by formulae \eqref{eq:Njchoice475} and the numerical quadrature formula \eqref{eq:quadrature} with step size $h=0.001$, i.e. $N=N_j=500$ steps.\end{proof}

\begin{lemma}\label{l:dthreeprime475} We have $d^{V}(t)<0$ for all $4.5 \leq t \leq 5$.
\end{lemma}
\begin{proof} We approximate $d^{V}(t)$ by the polynomial $P_{6}(t)$ constructed in \eqref{eq:Pndef475} as the approximate value of the order 6 Taylor polynomial of $d^{V}$ around $t_0:=4.75$. As the error is at most $\de=39.9$, it suffices to show that $p(t):=P_{6}(t)+\de<0$ in $[4.5,5]$. Now $P_{6}(4.5)=-39.9655627058...$ so $P_{6}(4.5)+\de<0$. Moreover, $p'(t)=P_{6}'(t)=\sum_{j=1}^{6} \dfrac{\overline{d}_j}{(j-1)!} (t-4.75)^{j-1}$ and $p'(4.5)=-174.8777051...<0$. From the explicit formula of $p(t)$ we consecutively compute also $p''(4.5)=-662.2069802...<0$, $p'''(4.5)=-2199.092624...<0$, $p^{(4)}(4.5)=-7810.957541...2<0$ and $p^{(5)}(4.5)=-12,951.20993...<0$. Finally, we arrive at $p^{(6)}(t)=\overline{d}_6$=-119,563.5221... We have already checked that $p^{(j)}(4.5)<0$ for $j=0\dots 5$, so in order to conclude $p(t)<0$ for $4.5\leq t\leq 5$ it suffices to show $p^{(6)}(t)<0$ in the given interval. However, $p^{(6)}$ is constant, so $p^{(6)}(t)<0$ for all $t \in \RR$.  It follows that also $p(t)<0$ for all $t\geq 4.5$.
\end{proof}

\section{Conclusion}\label{sec:final}

With the help of the sharper quadrature formula \eqref{eq:quadrature} further numerical analysis is possible for higher values of $k$. In principle we can divide the interval $(k,k+1)$ to smaller and smaller intervals to get improved error estimations of Taylor expansions to compensate the larger and larger error bounds resulting from e.g. \eqref{eq:Gpmadmnorm} and the increase of $t$. We have a strong feeling that this way we could work further to higher values of $k$. However, even that possibility does not mean that we would have a clear \emph{theoretical} reason, a firm grasp of the underlying law, rooted in the nature of the question, for what the result should hold for \emph{all} $k$.

\end{document}